\documentclass{amsart}
\usepackage{amssymb,latexsym,amsmath,amscd,graphicx,graphics,epic,eepic,bm,color,array,mathrsfs,fullpage}
\usepackage{enumerate}
\usepackage[all,knot,poly]{xy}

\newcommand{\zed}{\mathbb{Z}}

\newcommand{\Q}{\mathbb{Q}}

\newcommand{\im}{\mathrm{Im}}

\newcommand{\bn}[2]{\genfrac{(}{)}{0pt}{}{#1}{#2}}

\newcommand{\id}{\mathrm{id}}

\newcommand{\pd}{\mathrm{pd}}

\newcommand{\m}{\mathfrak{m}}

\newcommand{\Tor}{\mathrm{Tor}}

\newcommand{\slmf}{\mathfrak{sl}}

\theoremstyle{plain}
\newtheorem{theorem}{Theorem}[section]
\newtheorem{lemma}[theorem]{Lemma}
\newtheorem{proposition}[theorem]{Proposition}

\newtheorem{conjecture}[theorem]{Conjecture}

\theoremstyle{definition}
\newtheorem{definition}[theorem]{Definition}
\newtheorem{example}[theorem]{Example}

\theoremstyle{remark}
\newtheorem{remark}[theorem]{Remark}

\numberwithin{equation}{section}

\begin{document}

\title{Betti Numbers of the HOMFLYPT Homology}

\author{Hao Wu}

\thanks{The author was partially supported by NSF grant DMS-1205879.}

\address{Department of Mathematics, The George Washington University, Phillips Hall, Room 739, 801 22nd Street NW, Washington DC 20052, USA. Telephone: 1-202-994-0653, Fax: 1-202-994-6760}

\email{haowu@gwu.edu}

\subjclass[2010]{Primary 57M27}

\keywords{HOMFLYPT homology, Betti number, split link} 

\begin{abstract}
In \cite{Ras-2-bridge}, Rasmussen observed that the Khovanov-Rozansky homology of a link is a finitely generated module over the polynomial ring generated by the components of this link. In the current paper, we study the module structure of the middle HOMFLYPT homology, especially the Betti numbers of this module. For each link, these Betti numbers are supported on a finite subset of $\mathbb{Z}^4$. One can easily recover from these Betti numbers the Poincar\'e polynomial of the middle HOMFLYPT homology. We explain why the Betti numbers can be viewed as a generalization of the reduced HOMFLYPT homology of knots. As an application, we prove that the projective dimension of the middle HOMFLYPT homology is additive under split union of links and provides a new obstruction to split links.
\end{abstract}

\maketitle

\section{Introduction}\label{sec-intro}

In \cite{Ras-2-bridge}, Rasmussen observed that the $\slmf(N)$ homology of a link defined in \cite{KR1} is a finitely generated module over the polynomial ring generated by the components of this link. His observation applies to other versions of the Khovanov-Rozansky homology too. And it is not hard to see that the module structure of the Khovanov-Rozansky homology over this polynomial ring is a link invariant. In the current paper, we study the module structures of the middle HOMFLYPT homology $H$ defined in \cite[Definition 2.9]{Ras2} and its reduction $H_r$ with respect to one component.\footnote{The middle HOMFLYPT homology $H$ in the current paper is defined exactly as in \cite{Ras2}. However, its reduction $H_r$ defined in Section \ref{sec-HOMFLYPT-mod} below is different from the reduced HOMFLYPT homology $\overline{H}$ in \cite{Ras2}. For non-split links, both $H_r$ and $\overline{H}$ are the same quotient of $H$. But, for links with split diagrams, $H_r$ remains a quotient of $H$, while $\overline{H}$ follows a more complex definition in \cite[Section 2.10]{Ras2} and is no longer a quotient of $H$.} Please see Section \ref{sec-HOMFLYPT-mod} below for a brief review of $H$ and $H_r$, especially Lemmas \ref{lemma-HOMFLYPT-module} and \ref{lemma-HOMFLYPT-module-inv} for their module structures.

Let $B$ be a closed braid, and $K_1,\dots,K_m$ the components of $B$. To each $K_i$, we assign a homogeneous variable $X_i$ of degree $2$. Define graded rings $R_B:=\Q[X_1,\dots,X_m]$ and $R_{B,r}:=\Q[X_2-X_1,\dots,X_m-X_1]$.\footnote{Note that $R_{B,r}$ does note depend on the ordering of the components. In fact, for any $1\leq j \leq m$, ${R_{B,r}=\Q[X_1-X_j,\dots,X_{j-1}-X_j,X_{j+1}-X_j,\dots,X_m-X_j]}$.} Let $H^{\star,j,k}(B) = \oplus_{i\in \zed} H^{i,j,k}(B)$ and $H_r^{\star,j,k}(B) = \oplus_{i\in \zed} H_r^{i,j,k}(B)$. According to Lemma \ref{lemma-HOMFLYPT-module-inv}, $H^{\star,j,k}(B)$ (resp. $H_r^{\star,j,k}(B)$) is a $\zed$-graded $R_B$-module (resp. $R_{B,r}$-module.) Let $\m=(X_1,\dots,X_m)$ be the maximal homogeneous ideal of $R_B$, and $\m_r=(X_2-X_1,\dots,X_m-X_1)$ be the maximal homogeneous ideal of $R_{B,r}$. Note that $R_B/\m$ (resp. $R_{B,r}/\m_r$) is also a $\zed$-graded $R_B$-module (resp. $R_{B,r}$-module.) So $\Tor^{R_B}_p(R_B/\m,H^{\star,j,k}(B))$ and $\Tor^{R_{B,r}}_p(R_{B,r}/\m_r,H_r^{\star,j,k}(B))$ are both $\zed$-graded $\Q$-spaces.

\begin{definition}\label{def-Betti}
The Betti numbers of $H(B)$ and $H_r(B)$ are defined to be
\begin{eqnarray*}
\beta_B(p,q,j,k) & := & \dim_\Q \Tor^{R_B}_p(R_B/\m,H^{\star,j,k}(B))^q, \\
\beta_{B,r}(p,q,j,k) & := & \dim_\Q \Tor^{R_{B,r}}_p(R_{B,r}/\m_r,H_r^{\star,j,k}(B))^q,
\end{eqnarray*}
where $\Tor^{R_B}_p(R_B/\m,H^{\star,j,k}(B))^q$ (resp. $\Tor^{R_{B,r}}_p(R_{B,r}/\m_r,H_r^{\star,j,k}(B))^q$) is the homogeneous component of \linebreak $\Tor^{R_B}_p(R_B/\m,H^{\star,j,k}(B))$ (resp. $\Tor^{R_{B,r}}_p(R_{B,r}/\m_r,H_r^{\star,j,k}(B))$) of degree $q$.
\end{definition}

Clearly, $\beta_B(p,q,j,k)$ and $\beta_{B,r}(p,q,j,k)$ are defined for $(p,q,j,k) \in \zed_{\geq0}\times\zed^3$. 

The main technical tool we use to study the Betti numbers are minimal free resolutions. We will review these in Section \ref{sec-minimal-res} below. The following are some basic properties of the Betti numbers of the HOMFLYPT homology.

\begin{lemma}\label{lemma-Betti-inv}
\begin{enumerate}[1.]
  \item For every $(p,q,j,k) \in \zed_{\geq0}\times\zed^3$, $\beta_B(p,q,j,k)$ and $\beta_{B,r}(p,q,j,k)$ are invariant under Markov moves of $B$. 
	\item $\beta_B(p,q,j,k) = \beta_{B,r}(p,q,j,k) = 0$ for all but finitely many $(p,q,j,k) \in \zed_{\geq0}\times\zed^3$. 
\end{enumerate}
\end{lemma}

It is a standard fact that one can recover from the Betti numbers the graded dimension of a graded module over a polynomial ring. Based on this, we can easily recover the Poincar\'e polynomials of $H(B)$ and $H_r(B)$ from the Betti numbers $\beta_B(p,q,j,k)$ and $\beta_{B,r}(p,q,j,k)$. Let us first normalize the binomial numbers by
\begin{equation}\label{eq-binomial}
\bn{n}{k}= \begin{cases}
\frac{n!}{k!(n-k)!} &\text{if } 0\leq k \leq n, \\
1 & \text{if } n=-1 \text{ and } k=0, \\
0 &\text{otherwise.}
\end{cases}
\end{equation}

\begin{definition}\label{def-Betti-polynomial}
\begin{eqnarray*}
\mathcal{P}_B(x, y, a, b) & := & \sum_{(p,q,j,k) \in \zed_{\geq0}\times\zed^3} \beta_B(p,q,j,k) \cdot x^p \cdot (\sum_{i\in \zed} y^{2i+q} \cdot \bn{i+m-1}{i})\cdot a^j \cdot b^{\frac{k-j}{2}} \\
& = & \sum_{(p,q,j,k) \in \zed_{\geq0}\times\zed^3} \beta_B(p,q,j,k) \cdot x^p \cdot \frac{y^q}{(1-y^2)^m}\cdot a^j \cdot b^{\frac{k-j}{2}}
\end{eqnarray*}
and
\begin{eqnarray*}
\mathcal{P}_{B,r}(x, y, a, b) & := & \sum_{(p,q,j,k) \in \zed_{\geq0}\times\zed^3} \beta_{B,r}(p,q,j,k) \cdot x^p \cdot (\sum_{i\in \zed} y^{2i+q} \cdot \bn{i+m-2}{i})\cdot a^j \cdot b^{\frac{k-j}{2}} \\
& = & \sum_{(p,q,j,k) \in \zed_{\geq0}\times\zed^3} \beta_B(p,q,j,k) \cdot x^p \cdot \frac{y^q}{(1-y^2)^{m-1}}\cdot a^j \cdot b^{\frac{k-j}{2}},
\end{eqnarray*}
where $m$ is the number of components of the closed braid $B$.
\end{definition}

\begin{lemma}\label{lemma-Poincare-polynomials}
Polynomials $\mathcal{P}_B(x, y, a, b)$ and $\mathcal{P}_{B,r}(x, y, a, b)$ are invariant under Markov moves of $B$. Moreover, 
\begin{eqnarray*}
\mathcal{P}_B(-1, y, a, b) & = & \sum_{(i,j,k) \in \zed^3} y^i \cdot a^j \cdot b^{\frac{k-j}{2}} \cdot \dim_\Q H^{i,j,k}(B), \\
\mathcal{P}_{B,r}(-1, y, a, b) & = & \sum_{(i,j,k) \in \zed^3} y^i \cdot a^j \cdot b^{\frac{k-j}{2}} \cdot \dim_\Q H_r^{i,j,k}(B).
\end{eqnarray*}
Consequently, $\mathcal{P}_B(-1, y, a, -1) = -\frac{P_B(a,y)}{y-y^{-1}}$ and $\mathcal{P}_{B,r}(-1, y, a, -1) = P_B(a,y)$, where $P_B$ is a normalization of the HOMFLYPT polynomial.
\end{lemma}

\begin{remark}
Note that, for any non-vanishing homogeneous element of the middle HOMFLYPT homology, its first and second $\zed$-gradings always have the opposite parity. By Lemma \ref{lemma-Poincare-polynomials}, for fixed $(j,k)\in \zed^2$,
\begin{equation}\label{eq-hilbert}
\dim_\Q H^{2T+1-j,j,k}(B)= \sum_{(p,q)\in \zed_{\geq0}\times \zed} (-1)^p\cdot\beta_B(p,q,j,k)\cdot\bn{T+m-\frac{j+q+1}{2}}{T-\frac{j+q-1}{2}}.
\end{equation}
Note that, when $T \geq \max\{\frac{j+q-1}{2}~|~\beta_B(p,q,j,k)\neq 0\}$, the right hand side of Equation \eqref{eq-hilbert} is a polynomial of $T$. Thus, the Betti numbers determine when the Hilbert function of $H^{\star,j,k}(B)$ becomes its Hilbert polynomial. This was implicitly asked in \cite[Question 1.7]{Wu-hilbert}.
\end{remark}

It is another standard fact that one can recover from the Betti numbers the projective dimension of a graded module over a polynomial ring. For $H(B)$ and $H_r(B)$, we have the following lemma.

\begin{lemma}\label{lemma-pd}
\begin{enumerate}
	\item $\pd_{R_B} H(B)=\deg_x \mathcal{P}_B(x, y, a, b)$, where $\pd_{R_B} H(B)$ is the projective dimension of $H(B)$ over $R_B$.
	\item $\pd_{R_{B,r}} H_r(B)=\deg_x \mathcal{P}_{B,r}(x, y, a, b)$, where $\pd_{R_{B,r}} H_r(B)$ is the projective dimension of $H_r(B)$ over $R_{B,r}$.
\end{enumerate}
\end{lemma}

Lemmas \ref{lemma-Betti-inv}, \ref{lemma-Poincare-polynomials} and \ref{lemma-pd} will be proved in Section \ref{sec-Betti} below.

Our results start with the observation that the Betti numbers of the middle HOMFLYPT homology $H$ and its reduction $H_r$ are essentially the same. 

\begin{theorem}\label{thm-Betti}
$\beta_B(p,q,j,k) = \beta_{B,r}(p,q-1,j,k)$ for all $(p,q,j,k) \in \zed_{\geq0}\times\zed^3$. In particular, if $B$ is a knot, then $\beta_B(0,q,j,k) = \dim_\Q H_r^{q-1,j,k}(B) =\dim_\Q \overline{H}^{q-1,j,k}(B)$ and $\beta_B(p,q,j,k)=0$ whenever $p>0$, where $\overline{H}$ is the reduced HOMFLYPT homology defined in \cite{Ras2}.
\end{theorem}

Theorem \ref{thm-Betti} will also be proved in Section \ref{sec-Betti} below.

\begin{remark}
By Lemma \ref{lemma-Poincare-polynomials}, the Betti numbers of $H(B)$ determine the Poincar\'e polynomial of $H(B)$. By Theorem \ref{thm-Betti}, the Betti numbers of $H(B)$ further determine the Poincar\'e polynomial of $H_r(B)$. Comparing the definition of $H_r$ in Section \ref{sec-HOMFLYPT-mod} below and that of the reduced HOMFLYPT homology $\overline{H}$ in \cite{Ras2}, we know the Poincar\'e polynomial of $\overline{H}(B)$ is equal to that of $H_r(B)$ times $(1+a^{-2}b)^n$ for some $n \in \mathbb{N}_{\geq0}$. So, provided we know what $n$ is, the Betti numbers of $H(B)$ also determine the Poincar\'e polynomial of $\overline{H}(B)$.

Some researchers prefer to work with link homologies represented by a finite set of data. If $B$ is a knot, its reduced HOMFLYPT homology is finite dimensional, which is why these researchers prefer this version of the HOMFLYPT homology over others. Theorem \ref{thm-Betti} shows that, for knots, the Betti numbers are the dimensions of homogeneous components of the reduced HOMFLYPT homology. If $B$ is a link with multiple components, then its reduced HOMFLYPT homology becomes infinite dimensional. However, its Betti numbers remain a finite set of data. In this sense, the Betti numbers may play the same role for links as that played by the reduced HOMFLYPT homology for knots.
\end{remark}

It turns out that, up to a factor of $ab^{-1}$, the polynomial $\mathcal{P}_B(x, y, a, b)$ is multiplicative under split union of closed braids. So it follows from Lemma \ref{lemma-pd} that the projective dimension of $H(B)$ is additive under split union of closed braids. This leads to a new obstruction to split links. First, let us recall the definition of the split union of braids.

\begin{definition}\label{def-split}
Denote by $\mathbf{B}_k$ the braid group on $k$ strands with standard generators $\sigma_1^{\pm1},\dots,\sigma_{k-1}^{\pm1}$. Let $B_1$ and $B_2$ be closed braids with braid words $w_1=\sigma_{i_1}^{\mu_1}\cdots \sigma_{i_{l_1}}^{\mu_{l_1}} \in \mathbf{B}_{k_1}$ and $w_2=\sigma_{j_1}^{\nu_1}\cdots \sigma_{j_{l_2}}^{\nu_{l_2}} \in \mathbf{B}_{k_2}$, respectively. The split union $B_1 \sqcup B_2$ of $B_1$ and $B_2$ is the closed braid with the braid word $\sigma_{i_1}^{\mu_1}\cdots \sigma_{i_{l_1}}^{\mu_{l_1}}\sigma_{j_1+k_1}^{\nu_1}\cdots \sigma_{j_{l_2}+k_1}^{\nu_{l_2}} \in \mathbf{B}_{k_1+k_2}$.

Clearly, the operation of split union is associative. And it is commutative up to Markov moves. 

A closed braid $B$ is $n$-split if and only if there exist $n$ closed braids $B_1,\dots,B_n$ such that $B=B_1\sqcup\cdots\sqcup B_n$.

A link is $n$-split if and only if it is equivalent to an $n$-split closed braid. 
\end{definition}

One can see that every link is $1$-split. And a link is $2$-split if and only if it is split in the classical sense.

Now we can state our results on split links.

\begin{theorem}\label{thm-split}
\begin{enumerate}
	\item For any two closed braids $B_1$ and $B_2$, 
	\[
	\mathcal{P}_{B_1 \sqcup B_2}(x, y, a, b)= a\cdot b^{-1} \cdot \mathcal{P}_{B_1}(x, y, a, b) \cdot \mathcal{P}_{B_2}(x, y, a, b). 
	\]
	Consequently, $\pd_{R_{B_1 \sqcup B_2}} H(B_1 \sqcup B_2) = \pd_{R_{B_1}} H(B_1)+\pd_{R_{B_2}} H(B_2)$.
	\item If $B$ is a closed braid diagram of an $m$-component $n$-split link, then $\pd_{R_B} H(B) \leq m-n$.
\end{enumerate}
\end{theorem}

Theorem \ref{thm-split} will be proved in Section \ref{sec-split} below.

\begin{remark}
The distant from a link to being split is usually measure by the splitting number, that is, the minimal number of crossing changes needed to make the link split. Consider the $2$-strand closed braid $B_n$ with the braid word $\sigma_1^{2n}\in \mathbf{B}_2$. The splitting number of $B_n$ is $n$. But, by Theorem \ref{thm-split}, $0 \leq \pd_{R_{B_n}} H(B_n) \leq 1$. This example shows that $\pd_{R_B} H(B)$ is not a good indicator of how far a link is from being split. See \cite{Batson-Seed} for lower bounds of the splitting number from the Khovanov homology.

On the other hand, $\pd_{R_B} H(B)$ may turn out to be a good indicator of how many times we can split a link. Theorem \ref{thm-split} seems to suggest that, the more times we can split a link, the smaller the projective dimension of its HOMFLYPT homology gets in comparison to the number of components of this link. 
\end{remark}

\begin{conjecture}\label{conj-split}
An $m$-component closed braid $B$ represents an $n$-split link if and only if $\pd_{R_B} H(B) \leq m-n$.
\end{conjecture}

\begin{example}\label{eg-hopf}
Consider the positive Hopf link $B_1$ with braid word $\sigma_1^{2}\in \mathbf{B}_2$. In Section \ref{sec-hopf} below, we will prove that $\pd_{R_{B_1}} H(B_1) = 1$. By Theorem \ref{thm-split}, this confirms the well known fact that the Hopf link does not split. 
\end{example}

For a closed braid $B$, its $\slmf(N)$ homology $H_N(B)$ is also a module over $R_B$. But the projective dimension of $H_N(B)$ over $R_B$ is far less interesting.

\begin{lemma}\label{lemma-sl-N-proj-dim}
Let $B$ be a closed braid of $m$ components. Then, for any $N\geq 1$, $\pd_{R_B} H_N(B) =m$.
\end{lemma}

Lemma \ref{lemma-sl-N-proj-dim} will be proved in Section \ref{sec-sl-N} below.

\begin{remark}
Note that the $\slmf(N)$ homology $H_N(B)$ is not just a module over $R_B$. For each component $K_j$, the monomial $X_j^N$ acts on $H_N(B)$ as $0$. So $H_N(B)$ is actually a module over the quotient ring $R_{B,N}:= \Q[X_1,\dots,X_m]/(X_1^N,\dots,X_m^N)$. Since $R_{B,N}$ is a local ring, techniques based on minimal free resolutions should still work. It would be interesting to see what topological information the Betti numbers of $H_N(B)$ over $R_{B,N}$ contain.
\end{remark}

\section{Module Structure of the HOMFLYPT Homology}\label{sec-HOMFLYPT-mod}

In this section, we briefly review the middle HOMFLYPT homology $H$ defined in \cite{Ras2} and its reduction $H_r$. For more details, see \cite{Ras2}.

\subsection{Base rings of chain complexes} For a closed braid, an edge of it is a part of the closed braid that starts and ends at crossings, but contains no crossings in its interior. In the rest of this section, we fix a closed braid $B$ with $m$ components $K_1,\dots,K_m$. We order the edges of $B$ as $1^{st},2^{nd},\dots,M^{th}$ so that the $l^{th}$ edge is on the component $K_l$ for $1 \leq l \leq m$. For $1 \leq l \leq M$, we assign a variable $X_l$ of degree $2$ to the $l^{th}$ edge. For a crossing $c$ of $B$, assume the $k^{th}$ and $l^{th}$ edges are pointing out of $c$, and the $i^{th}$ and $j^{th}$ edges are pointing into $c$. Then $c$ defines a relation $\rho(c)=X_k+X_l-X_i-X_j$. The edge ring of $B$ is the ring $R(B):=\Q[X_1,\dots,X_M]/(\rho(c_1),\dots,\rho(c_n))$, where $c_1,\dots,c_n$ are all the crossings of $B$. The reduced edge ring of $B$ is the ring $R_r(B):=\Q[X_2-X_1,\dots,X_M-X_1]/(\rho(c_1),\dots,\rho(c_n))$. One can see that $R_r(B)$ is a subring of $R(B)$. Moreover,
\begin{equation}\label{eq-ring-tensor}
R(B) = R_r(B) \otimes_\Q \Q[X_1].
\end{equation} 
Note that $R(B)$ and $R_r(B)$ are not the rings $R_B$ and $R_{B,r}$ defined in the introduction.

\subsection{HOMFLYPT homologies} As defined in \cite{Ras2}, the middle complex $(C_0(B), d_+, d_v)$ is a $\zed^3$-graded double cochain complex\footnote{Strictly speaking, $C_0(B)$ is not a double complex since squares in it commute, instead of anti-commute. But this does not affect any of our computations.} of finitely generated graded free $R(B)$-modules with homogeneous differential maps. Its first grading is the grading of the underlying $R(B)$-module. Its second and third gradings are the horizontal and vertical gradings of the double complex. These two gradings are both bounded. The reduced complex $(C_r(B), d_+, d_v)$ is defined by replacing each summand of $R(B)$ in $C_0(B)$ by a summand of $R_r(B)$. Clearly, $(C_r(B), d_+, d_v)$ is a $\zed^3$-graded double cochain complex of finitely generated graded free $R_r(B)$-module with homogeneous differential maps. By \cite[Lemma 2.12]{Ras2}, we know that
\begin{eqnarray}
\label{eq-complex-quotient} C_r(B) & \cong & C_0(B) / X_1 C_0(B), \\
\label{eq-complex-tensor} C_0(B) & \cong & C_r(B) \otimes_\Q \Q[X_1].
\end{eqnarray} 
Note that $C_0(B)$ (resp. $C_r(B)$) is a finitely generated module over $R(B)$ (resp. $R_r(B)$.)

The middle HOMFLYPT homology $H$ and its reduction $H_r$ are defined by
\begin{eqnarray}
\label{eq-H-def} H(B) & := & H(H(C_0(B), d_+), d_v)\{-w+b,w+b-1,w-b+1\},\\
\label{eq-Hbar-def} H_r(B) & := & H(H(C_r(B), d_+), d_v)\{-w+b-1,w+b-1,w-b+1\},
\end{eqnarray}
where $w$ is the writhe of $B$, $b$ is the number of strands in $B$, and ``$\{s, t, u\}$" means shifting the $\zed^3$-grading by the vector $(s, t, u)$. Note here that the definition of $H_r(B)$ is different from that of $\overline{H}(B)$ in \cite{Ras2}. The difference occurs when the closed braid $B$ splits. See \cite[Section 2.10]{Ras2}.

It is proved in \cite{KR2} that $H(B)$ and $H_r(B)$ are invariant as $\zed^3$-graded $\Q$-spaces under Markov moves of $B$.

One of the main advantages of the middle HOMFLYPT homology over the other normalizations of the HOMFLYPT homology is that, up to a grading shift, it is tensorial over $\Q$ under the split union. More precisely, let $B_1$ and $B_2$ be two closed braids. Then
\begin{eqnarray}
\label{eq-C-tensor} C_0(B_1\sqcup B_2) & \cong & C_0(B_1) \otimes_\Q C_0(B_2), \\
\label{eq-H-tensor} H(B_1\sqcup B_2) & \cong & H(B_1) \otimes_\Q H(B_2)\{0,1,-1\},
\end{eqnarray}
where, of course, the isomorphisms preserve the $\zed^3$-grading.

\subsection{Module structures of the HOMFLYPT homologies} Since $R(B)$ (resp. $R_r(B)$) is Noetherian, $H(B)$ (resp. $H_r(B)$) is a finitely generated module over $R(B)$ (resp. $R_r(B)$.) But there is no chance for these module structures to be invariant under Markov moves. This is simply because $R(B)$ and $R_r(B)$ change under Markov moves. But, in \cite[Lemma 3.4]{Ras-2-bridge}, Rasmussen observed that, if $X_k$ and $X_l$ are assigned to edges on the same component of $B$, then their actions on $H(B)$ are the same.\footnote{\cite[Lemma 3.4]{Ras-2-bridge} is about the $\slmf(N)$ homology. But its conclusion and proof remain true for the HOMFLYPT homology.} So $H(B)$ is a finitely generated module over the quotient ring 
\[
R(B)/(\{X_k-X_l~|~\text{the } k^{th} \text{ and } l^{th} \text{ edges are on the same component of }B\}) \cong R_B = \Q[X_1,\dots,X_m].
\]
Similar conclusion holds for $H_r(B)$. We have the following lemma.

\begin{lemma}\cite[Lemma 3.4]{Ras-2-bridge}\label{lemma-HOMFLYPT-module}
Let $B$ be a closed braid diagram, and $K_1,\dots,K_m$ be the components of $B$. To each $K_i$, we assign a variable $X_i$ with degree $2$. Then:
\begin{itemize}
	\item $H(B)$ is a finitely generated $\zed^3$-graded module over the $\zed$-graded ring $R_B:=\Q[X_1,\dots,X_m]$, where the action of any homogeneous element of $R_B$ on $H(B)$ fixes the last two $\zed$-gradings of $H(B)$, but shifts the first by its own degree.
	\item $H_r(B)$ is a finitely generated $\zed^3$-graded module over the $\zed$-graded ring $R_{B,r}:={\Q[X_2-X_1,\dots,X_m-X_1]}$, where the action of any homogeneous element of $R_{B,r}$ on $H_r(B)$ fixes the last two $\zed$-gradings of $H_r(B)$, but shifts the first by its own degree.
\end{itemize}
\end{lemma}

\begin{proof}
The proof of this lemma is a straightforward adaptation of the proof of \cite[Lemma 3.4]{Ras-2-bridge}. We leave the details to the reader.
\end{proof}

\begin{remark}\label{remark-quotient}
Applying the Universal Coefficient Theorem over $\Q$ to the right hand side of \eqref{eq-complex-tensor}, we get that $H(B) \cong H_r(B) \otimes_\Q \Q[X_1]\{1,0,0\}$ as $\zed^3$-graded $R_B$-modules. Note that $R_{B,r}\cong R_B/(X_1)$ with the isomorphism given by $R_{B,r} \hookrightarrow R_B \twoheadrightarrow R_B/(X_1)$, where ``$\hookrightarrow$" is the standard inclusion, and ``$\twoheadrightarrow$" is the standard quotient map. Identify $R_{B,r}$ and $R_B/(X_1)$ via this isomorphism. Then $H_r(B) \cong H(B)/X_1H(B)\{-1,0,0\}$ as $\zed^3$-graded $R_{B,r}$-modules.
\end{remark}

Next, we show that the module structures of $H(B)$ and $H_r(B)$ in Lemma \ref{lemma-HOMFLYPT-module} are invariant under Markov moves.

\begin{figure}[ht]
$
\xymatrix{
\setlength{\unitlength}{1pt}
\begin{picture}(60,40)(-30,0)

\put(-10,10){\vector(1,1){30}}

\put(20,0){\line(-1,1){18}}

\put(-2,22){\line(-1,1){8}}

\qbezier(-10,10)(-20,0)(-20,20)

\qbezier(-10,30)(-20,40)(-20,20)

\multiput(-25,35)(5,0){7}{\line(1,0){2}}

\multiput(-25,5)(5,0){7}{\line(1,0){2}}

\multiput(-25,5)(0,5){6}{\line(0,1){2}}

\multiput(10,5)(0,5){6}{\line(0,1){2}}

\end{picture} \ar@<5ex>[rr] && \setlength{\unitlength}{1pt}
\begin{picture}(60,40)(-30,0)

\put(5,25){\vector(1,1){15}}

\put(20,0){\line(-1,1){15}}

\qbezier(5,15)(0,20)(5,25)

\multiput(-25,35)(5,0){7}{\line(1,0){2}}

\multiput(-25,5)(5,0){7}{\line(1,0){2}}

\multiput(-25,5)(0,5){6}{\line(0,1){2}}

\multiput(10,5)(0,5){6}{\line(0,1){2}}

\end{picture} \ar@<-4ex>[ll] \ar@<5ex>[rr] && \setlength{\unitlength}{1pt}
\begin{picture}(60,40)(-30,0)

\put(-10,10){\line(1,1){8}}

\put(2,22){\vector(1,1){18}}

\put(20,0){\line(-1,1){30}}

\qbezier(-10,10)(-20,0)(-20,20)

\qbezier(-10,30)(-20,40)(-20,20)

\multiput(-25,35)(5,0){7}{\line(1,0){2}}

\multiput(-25,5)(5,0){7}{\line(1,0){2}}

\multiput(-25,5)(0,5){6}{\line(0,1){2}}

\multiput(10,5)(0,5){6}{\line(0,1){2}}

\end{picture} \ar@<-4ex>[ll] \\
\setlength{\unitlength}{1pt}
\begin{picture}(50,60)(-25,0)

\put(-10,0){\vector(0,1){60}}

\put(10,0){\vector(0,1){60}}

\multiput(-15,50)(5,0){6}{\line(1,0){2}}

\multiput(-15,10)(5,0){6}{\line(1,0){2}}

\multiput(-15,10)(0,5){8}{\line(0,1){2}}

\multiput(15,10)(0,5){8}{\line(0,1){2}}

\end{picture} \ar@<7ex>[rr] && \setlength{\unitlength}{1pt}
\begin{picture}(50,60)(-25,0)

\put(-15,0){\line(1,1){15}}

\qbezier(0,15)(15,30)(0,45)

\put(0,45){\vector(-1,1){15}}

\put(15,0){\line(-1,1){13}}

\put(2,47){\vector(1,1){13}}

\qbezier(-2,17)(-15,30)(-2,43)

\multiput(-15,50)(5,0){6}{\line(1,0){2}}

\multiput(-15,10)(5,0){6}{\line(1,0){2}}

\multiput(-15,10)(0,5){8}{\line(0,1){2}}

\multiput(15,10)(0,5){8}{\line(0,1){2}}

\end{picture} \ar@<-6ex>[ll] && \\
\setlength{\unitlength}{1pt}
\begin{picture}(90,60)(-45,0)

\put(-30,0){\vector(1,1){60}}

\put(30,0){\line(-1,1){28}}

\put(-2,32){\line(-1,1){11}}

\put(-17,47){\vector(-1,1){13}}

\put(0,0){\line(-1,1){13}}

\put(-15,45){\vector(1,1){15}}

\qbezier(-17,17)(-30,30)(-15,45)

\multiput(-25,50)(5,0){8}{\line(1,0){2}}

\multiput(-25,10)(5,0){8}{\line(1,0){2}}

\multiput(-25,10)(0,5){8}{\line(0,1){2}}

\multiput(15,10)(0,5){8}{\line(0,1){2}}

\end{picture} \ar@<7ex>[rr] && \setlength{\unitlength}{1pt}
\begin{picture}(90,60)(-45,0)

\put(-30,0){\vector(1,1){60}}

\put(30,0){\line(-1,1){13}}

\put(13,17){\line(-1,1){11}}

\put(-2,32){\vector(-1,1){28}}

\put(0,0){\line(1,1){15}}

\put(13,47){\vector(-1,1){13}}

\qbezier(15,15)(30,30)(17,43)

\multiput(-15,50)(5,0){8}{\line(1,0){2}}

\multiput(-15,10)(5,0){8}{\line(1,0){2}}

\multiput(-15,10)(0,5){8}{\line(0,1){2}}

\multiput(25,10)(0,5){8}{\line(0,1){2}}

\end{picture} \ar@<-6ex>[ll] &&
}
$
\caption{Parts of the braid diagram changed by braid-like Reidemeister moves}\label{fig-Reidemeister-boxes}

\end{figure}
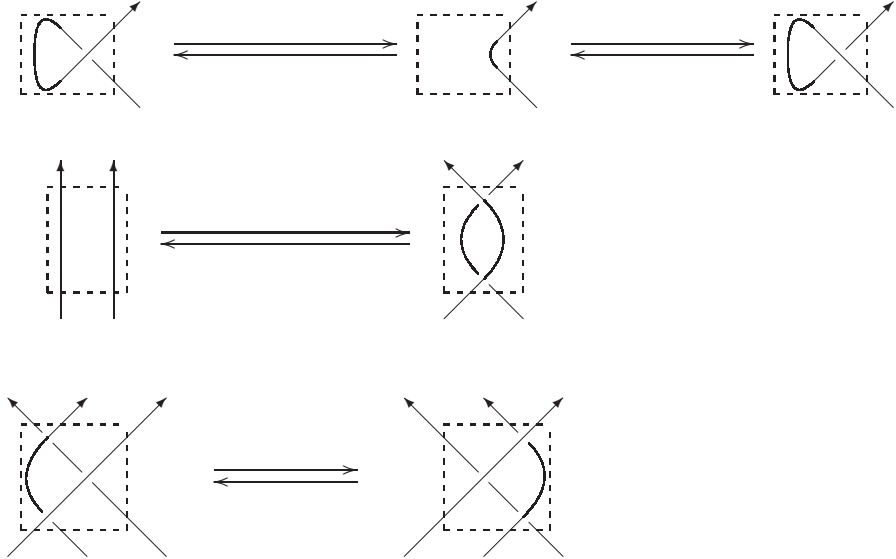

\begin{lemma}\label{lemma-HOMFLYPT-module-inv}
Assume that $B'$ is another closed braid diagram of the same link. Fix a sequence of Markov moves that changes $B$ to $B'$. Denote by $K_i'$ the component of $B'$ that is identified to $K_i$ through this sequence of Markov moves. To each $K_i'$, we assign a variable $X_i'$ of degree $2$. Set $R_{B'}:=\Q[X_1',\dots,X_m']$ and $R_{B',r}:=\Q[X_2'-X_1',\dots,X_m'-X_1']$. We identify the rings $R_B$ with $R_{B'}$ (resp. $R_{B,r}$ with $R_{B',r}$) via the equations $X_i=X_i'$ (resp. $X_i-X_1=X_i'-X_1'$.) Then this sequence of Markov moves induces:
\begin{itemize}
	\item an isomorphism $H(B)\cong H(B')$ of $\zed^3$-graded $R_B$-modules,
	\item an isomorphism $H_r(B)\cong H_r(B')$ of $\zed^3$-graded $R_{B,r}$-modules.
\end{itemize}
\end{lemma}

\begin{proof}
We only need to prove this lemma in the case when $B$ and $B'$ differ by a single braid-like Reidemeister move. For each braid-like Reidemeister move, Khovanov and Rozansky constructed in \cite{KR2} a $\Q$-linear isomorphism of $H(B)$ and $H(B')$ preserving the $\zed^3$-grading. This isomorphism commutes with the actions of the variables assigned to edges that are \emph{not entirely with in} the part of $B$ and $B'$ changed by the braid-like Reidemeister move, that is, \emph{not entirely with in} one of the dashed boxes in Figure \ref{fig-Reidemeister-boxes}. Note that every component of $B$ and $B'$ contains an edge not entirely with in this dashed box. With out loss of generality, we choose the variables assigned to each pair of corresponding components of $B$ and $B'$ to be the variables assigned to a pair of corresponding edges on these components that are not entirely with in this dashed box. Then Khovanov and Rozansky's isomorphism commutes with the variables assigned to all components of $B$ and $B'$. Thus, this isomorphism is an isomorphism of $R_B$-modules. This proves $H(B)\cong H(B')$ as $\zed^3$-graded $R_B$-modules. $H_r(B)\cong H_r(B')$ follows from $H(B)\cong H(B')$ and Remark \ref{remark-quotient}.
\end{proof}

\begin{remark}
Note that we did not claim the naturality of the isomorphisms in Lemma \ref{lemma-HOMFLYPT-module-inv}. We do not need the naturality for our results.
\end{remark}

\section{Minimal Free Resolutions}\label{sec-minimal-res}

Betti numbers of a module are often understood through the minimal free resolution of this module. In this section, we review basics of minimal free resolutions of graded modules over a polynomial ring. For more details, see for example \cite[Chapter 1]{Eisenbud-book-2}.

Let $R=\Q[X_1,\dots,X_m]$ be a polynomial ring graded by $\deg X_j =2$ for all $j=1,\dots,m$. The maximal homogeneous ideal of $R$ is $\m=(X_1,\dots,X_m)$.

\begin{theorem}[Hilbert's Syzygy Theorem]\label{thm-syzygy}
Assume that $M$ is a finitely generated graded $R$-module. Then there is a graded free resolution 
\[
0 \rightarrow F_l \rightarrow F_{l-1} \rightarrow \cdots \rightarrow F_1 \rightarrow F_0
\]
of $M$ over $R$, in which each $F_j$ is finitely generated over $R$, each arrow preserves the module grading, and $l \leq m$.
\end{theorem}

For a detailed elementary proof of Hilbert's Syzygy Theorem, see for example \cite[Theorem 4.3]{Arrondo-notes}.

\begin{definition}\label{def-minimal-res}
A chain complex of graded $R$-modules $\cdots \rightarrow C_p \xrightarrow{d_p} C_{p-1} \rightarrow \cdots$
is called minimal if $d_P(C_p) \subset \m C_{p-1}$ for each $p$.

A graded free resolution of a graded $R$-module is called a minimal free resolution if it is also a minimal chain complex of graded $R$-modules.
\end{definition}

\begin{theorem}\cite[Theorem 1.6]{Eisenbud-book-2}\label{thm-minimal-res}
If $M$ is a finitely generated graded $R$-module, then any finitely generated graded free resolution of $M$ over $R$ contains a minimal free resolution of $M$ as a direct summand. Moreover, any two minimal free resolutions of $M$ over $R$ are isomorphic as chain complexes of graded $R$-modules via an isomorphism that induces the identity map on $M$.
\end{theorem}

Clearly, Theorem \ref{thm-syzygy} and the first half of Theorem \ref{thm-minimal-res} guarantee the existence of the minimal free resolution of any graded $R$-module. The second half of Theorem \ref{thm-minimal-res} gives the uniqueness of the the minimal free resolution. 

The proof of the existence part of Theorem \ref{thm-minimal-res} is quite elementary. Say, $0 \rightarrow F_l \rightarrow F_{l-1} \rightarrow \cdots \rightarrow F_1 \rightarrow F_0$ is a finitely generated graded free resolution of $M$ over $R$. Fix a homogeneous $R$-basis for each $F_p$. Then each map $F_p \rightarrow F_{p-1}$ is given by a matrix whose entries are all homogeneous elements of $R$. Clearly, this resolution is minimal if and only if, for every $1\leq p \leq l$, all entries of this matrix are in $\m$. If this is not true, then, for some $p$, the matrix contains non-zero scalar $c$. Using this entry $c$, one can perform a change of bases for $F_p$ and $F_{p-1}$ to show that the original resolution has a direct summand of the form $0\rightarrow R \xrightarrow{c} R \rightarrow 0$. Removing this direct summand, we get a new ``smaller" graded free resolution of $M$. Repeat this process till there are no more non-zero scalars in the matrices representing the boundary maps. Then we get a minimal free resolution of $M$ that is a direct summand of the original graded free resolution of $M$. 

The proof of the uniqueness part of Theorem \ref{thm-minimal-res} requires some basic knowledge of homological algebra. It can be found in for example \cite[Theorem 20.2]{Eisenbud-book-1}. Note that the proof in \cite{Eisenbud-book-1} is for modules over local rings. But, with minor modifications, this proof also works for graded modules over polynomial rings. We do not actually use the uniqueness of the minimal free resolution in this paper.

\begin{definition}\label{def-Betti-general}
For a finitely generated graded $R$-module $M$, its $(p,q)^{th}$ Betti number is $\beta_M(p,q) := \dim_\Q \Tor^R_p(R/\m,M)^q$, where $\Tor^R_p(R/\m,M)^q$ is the homogeneous component of $\Tor^R_p(R/\m,M)$ of degree $q$. 
\end{definition}

The following lemma describes the relations between the minimal free resolution, Betti numbers and the projective dimension.

\begin{lemma}\cite[Proposition 1.7]{Eisenbud-book-2}\label{lemma-mfs-Betti}
Let $M$ be a finitely generated graded $R$-module, and $0 \rightarrow F_l \rightarrow F_{l-1} \rightarrow \cdots \rightarrow F_1 \rightarrow F_0$ a minimal free resolution of $M$ over $R$. Then, for every $p\geq 0$, as graded $\Q$-spaces,
\begin{equation}\label{eq-iso-Tor}
(R/\m)\otimes_R F_p \cong \Tor^R_p(R/\m,M).
\end{equation} 

Consequently,
\begin{itemize}
	\item any homogeneous $R$-basis for $F_p$ contains exactly $\beta_M(p,q)$ elements of degree $q$,
	\item the projective dimension of $M$ over $R$ is $\pd_R M = \max\{p~|~\beta_M(p,q)\neq 0 \text{ for some } q \in \zed.\}$
\end{itemize}
\end{lemma}

\begin{proof}
Recall that $\Tor^R_p(R/\m,M)$ is the $p^{th}$ homology of the chain complex 
\begin{equation}\label{eq-minimal-0-arrows}
0 \rightarrow (R/\m)\otimes_R F_l \rightarrow (R/\m)\otimes_R F_{l-1} \rightarrow \cdots \rightarrow (R/\m)\otimes_R F_1 \rightarrow (R/\m)\otimes_R F_0\rightarrow 0.
\end{equation}
The free resolution of $M$ being minimal implies that all arrows in the chain complex \eqref{eq-minimal-0-arrows} are zero maps. This proves isomorphism \eqref{eq-iso-Tor}.

The number of elements of degree $q$ in any homogeneous $R$-basis of $F_p$ is equal to the dimension over $\Q$ of the homogeneous component of $(R/\m)\otimes_R F_p$ of degree $q$, which, according to isomorphism \eqref{eq-iso-Tor}, is equal to $\dim_\Q \Tor^R_p(R/\m,M)^q = \beta_M(p,q)$.

The projective dimension of $M$ satisfies the inequality
\[
\max\{p~|~\Tor^R_p(R/\m,M)\neq 0\} \leq \pd_R M \leq \max\{p~|~F_p\neq 0.\}
\]
But, by isomorphism \eqref{eq-iso-Tor}, 
\[
\max\{p~|~F_p\neq 0\} = \max\{p~|~\Tor^R_p(R/\m,M)\neq 0\} = \max\{p~|~\beta_M(p,q)\neq 0 \text{ for some } q \in \zed.\}
\]
So $\pd_R M = \max\{p~|~\beta_M(p,q)\neq 0 \text{ for some } q \in \zed.\}$
\end{proof}

The following lemma explains how to recover the graded dimension of a graded $R$-module using its Betti numbers.

\begin{lemma}\label{lemma-Poincare-polynomial-general}
Let $M$ be a finitely generated graded $R$-module, and $0 \rightarrow F_l \rightarrow F_{l-1} \rightarrow \cdots \rightarrow F_1 \rightarrow F_0$ a minimal free resolution of $M$ over $R$. Denote by $M^i$ the homogeneous component of $M$ of degree $i$, and by $F_p^i$ the homogeneous component of $F_p$ of degree $i$. Then, for $p\geq 0$, 
\begin{equation}\label{eq-gdim-Fp}
\sum_{i \in \zed} y^i \cdot \dim_\Q F_p^i = \sum_{q\in \zed}\beta_M(p,q)\cdot (\sum_{i \in \zed} y^{2i+q} \bn{i+m-1}{i}).
\end{equation}
Consequently, 
\begin{equation}\label{eq-gdim-M}
\sum_{i \in \zed} y^i \cdot \dim_\Q M^i = \sum_{(p,q)\in \zed_{\geq0}\times\zed}(-1)^p\cdot \beta_M(p,q)\cdot (\sum_{i \in \zed} y^{2i+q} \bn{i+m-1}{i}).
\end{equation}
\end{lemma}

\begin{proof}
By Lemma \ref{lemma-mfs-Betti}, $F_p \cong \bigoplus_{q\in \zed} R\{q\}^{\oplus \beta_M(p,q)},$ where $R\{q\}$ is $R$ with grading raised by $q$. That is, the scalar $1$ in $R\{q\}$ has grading $q$. Note that $R\{q\}$ has graded dimension $\sum_{i \in \zed} y^{2i+q} \bn{i+m-1}{i}$. (Here, recall that each $X_j$ is of degree $2$.) This implies equation \eqref{eq-gdim-Fp}. But the graded dimension of $M$ is the alternating sum of the graded dimensions of $F_p$'s. Thus, we have equation \eqref{eq-gdim-M}.
\end{proof}

\section{Betti Numbers}\label{sec-Betti}

In this section, we prove Lemmas \ref{lemma-Betti-inv}, \ref{lemma-Poincare-polynomials}, \ref{lemma-pd} and Theorem \ref{thm-Betti}.

\begin{proof}[Proof of Lemmas \ref{lemma-Betti-inv}, \ref{lemma-Poincare-polynomials} and \ref{lemma-pd}]
For Lemma \ref{lemma-Betti-inv}, the invariance of the Betti numbers follows from Lemma \ref{lemma-HOMFLYPT-module-inv}. Since $H(B)$ (resp. $H_r(B)$) is finitely generated over $R_B$ (resp. $R_{B,r}$,) $\beta_B(p,q,j,k)$ (resp. $\beta_{B,r}(p,q,j,k)$) is non-zero for only finitely many $(p,q,j,k) \in \zed_{\geq0}\times\zed^3$.

For Lemma \ref{lemma-Poincare-polynomials}, polynomials $\mathcal{P}_B(x, y, a, b)$ and $\mathcal{P}_{B,r}(x, y, a, b)$ are invariant under Markov moves because their coefficients are invariant under Markov moves. Equations in this lemma follows from Lemma \ref{lemma-Poincare-polynomial-general}.

Lemma \ref{lemma-pd} follows from Lemma \ref{lemma-mfs-Betti}.
\end{proof}

It remains to prove Theorem \ref{thm-Betti}. To do this, we use the following graded version of \cite[Theorem 10.59]{Rotman-book}, which is a special case of the Grothendieck Spectral Sequence \cite[Theorem 10.48]{Rotman-book}. 

\begin{theorem}\label{thm-SS-base-ring}
Assume that:
\begin{itemize}
	\item $R$ and $S$ are graded Noetherian $\Q$-algebras;
	\item $A$ is a finitely generated graded right $R$-module;
	\item 
	\begin{itemize}
		\item $B$ is a left $R$-module and a right $S$-module,
		\item $B$ has a grading that makes it a graded left $R$-module and a graded right $S$-module;
	\end{itemize} 
	\item $\Tor^R_i(A,B\otimes_S P) \cong 0$ for all $i\geq 1$ whenever $P$ is a projective left $S$-module.
\end{itemize}
Then, for every finitely generated graded left $S$-module $C$, there is a first quadrant spectral sequence $\{E^r_{p,q}\}$ of graded $\Q$-spaces with $E^2_{p,q} \cong \Tor^R_p(A, \Tor^S_q(B,C))$ that converges to $\Tor^S_{\ast}(A\otimes_R B, C)$.
\end{theorem}

For the convenience of the reader, we include a proof of Theorem \ref{thm-SS-base-ring}. For this purpose, we need the following well known lemma.

\begin{lemma}\label{lemma-universal-projective}
Let $R$ be a graded $\Q$-algebra, and $Q$ a graded flat right $R$-module. Given any chain complex $(C_\ast,d)$ of graded left $R$-modules, we have $Q \otimes_R H_n(C_\ast) \cong H_n(Q \otimes_R C_\ast)$ as graded $\Q$-spaces for each $n$, where the isomorphism is the $\Q$-linear map given by $x\otimes[c] \mapsto [x\otimes c]$ for $x \in Q$ and $c \in \ker d_{n}$.
\end{lemma}

\begin{proof}
First consider the short exact sequence $0 \rightarrow \im d_{n+1} \xrightarrow{\jmath_{n}} \ker d_n \xrightarrow{\pi_n} H_n(C\ast) \rightarrow 0$ of graded left $R$-modules, where $\pi_n$ is the standard quotient map, and $\jmath_{n}$ is the standard inclusion. Since $Q$ is flat, we get a short exact sequence
\[
0 \rightarrow Q \otimes_R \im d_{n+1} \xrightarrow{\id_Q\otimes \jmath_{n}} Q\otimes_R\ker d_n \xrightarrow{\id_Q\otimes\pi_n} Q \otimes_R H_n(C\ast) \rightarrow 0
\]
of graded $\Q$-spaces. So we have $(Q\otimes_R\ker d_n)/(\id_Q\otimes \jmath_{n})(Q \otimes_R \im d_{n+1}) \cong Q \otimes_R H_n(C\ast)$ as graded $\Q$-spaces, where the isomorphism is given by $x \otimes c + (\id_Q\otimes \jmath_{n})(Q \otimes_R \im d_{n+1}) \mapsto x\otimes [c]$ for $x \in Q$ and $c \in \ker d_{n}$.

Now consider the short exact sequence $0 \rightarrow \ker d_n \xrightarrow{\iota_n} C_n \xrightarrow{d_n} \im d_n \rightarrow 0$ of graded left $R$-modules, where $\iota_n$ is the standard inclusion. Since $Q$ is flat, this gives us a short exact sequence
\[
0 \rightarrow Q\otimes_R\ker d_n \xrightarrow{\id_Q\otimes\iota_n} Q\otimes_R C_n \xrightarrow{\id_Q\otimes d_n} Q\otimes_R \im d_n \rightarrow 0
\]
of graded $\Q$-spaces, which induces a long exact sequence 
\[
\cdots \rightarrow Q \otimes_R \im d_{n+1} \xrightarrow{\Delta_{n+1}} Q\otimes_R\ker d_n \xrightarrow{(\id_Q\otimes\iota)_\ast} H_n(Q\otimes_R C_\ast)  \rightarrow Q \otimes_R \im d_{n} \xrightarrow{\Delta_{n}} \cdots
\]
of graded $\Q$-spaces. A simple diagram chase gives that the connecting homomorphism is $\Delta_n=\id_Q \otimes \jmath_{n-1}$, which, as shown above, is injective. Thus, we get a short exact sequence 
\[
0 \rightarrow Q \otimes_R \im d_{n+1} \xrightarrow{\id_Q\otimes \jmath_{n}} Q\otimes_R\ker d_n \xrightarrow{(\id_Q\otimes\iota)_\ast} H_n(Q\otimes_R C_\ast)  \rightarrow 0
\] 
of graded $\Q$-spaces. So we have $(Q\otimes_R\ker d_n)/(\id_Q\otimes \jmath_{n})(Q \otimes_R \im d_{n+1}) \cong H_n(Q\otimes_R C_\ast)$ as graded $\Q$-spaces, where the isomorphism is given by $x \otimes c + (\id_Q\otimes \jmath_{n})(Q \otimes_R \im d_{n+1}) \mapsto [x\otimes c]$ for $x \in Q$ and $c \in \ker d_{n}$.

Thus, $Q \otimes_R H_n(C_\ast) \cong H_n(Q \otimes_R C_\ast)$ as graded $\Q$-spaces, where the isomorphism is the $\Q$-linear map given by $x\otimes[c] \mapsto [x\otimes c]$ for $x \in Q$ and $c \in \ker d_{n}$.
\end{proof}

Now we are ready to prove Theorem \ref{thm-SS-base-ring}.

\begin{proof}[Proof of Theorem \ref{thm-SS-base-ring}]
Let $\cdots\rightarrow Q_1 \rightarrow Q_0$ be a graded projective resolution of the right $R$-module $A$, and $\cdots\rightarrow P_1 \rightarrow P_0$ be a graded projective resolution of the left $S$-module $C$. Consider the first quadrant double complex
\begin{equation}\label{eq-double-complex}
\xymatrix{
\cdots \ar[d]& \cdots \ar[d]  & \cdots \ar[d]  &   \\
Q_0 \otimes_R B \otimes_S P_2 \ar[d] & Q_1 \otimes_R B \otimes_S P_2 \ar[d] \ar[l] & Q_2 \otimes_R B \otimes_S P_2 \ar[d]\ar[l]  & \cdots \ar[l] \\
Q_0 \otimes_R B \otimes_S P_1 \ar[d] & Q_1 \otimes_R B \otimes_S P_1 \ar[d] \ar[l] & Q_2 \otimes_R B \otimes_S P_1 \ar[d]  \ar[l]& \cdots \ar[l] \\
Q_0 \otimes_R B \otimes_S P_0  & Q_1 \otimes_R B \otimes_S P_0 \ar[l] & Q_2 \otimes_R B \otimes_S P_0 \ar[l]  & \cdots \ar[l]
}
\end{equation}
Denote by $\{\hat{E}^r\}$ the spectral sequence of double complex \eqref{eq-double-complex} induced by its row filtration and by $\{E^r\}$ the spectral sequence of double complex \eqref{eq-double-complex} induced by its column filtration. Both of these are spectral sequences of graded $\Q$-spaces converging to the homology of the total complex of double complex \eqref{eq-double-complex}. 

First we consider the spectral sequence $\{\hat{E}^r\}$. Note that
\[
\hat{E}^0=\xymatrix{
\cdots & \cdots   & \cdots   &   \\
Q_0 \otimes_R (B \otimes_S P_2)  & Q_1 \otimes_R (B \otimes_S P_2) \ar[l] & Q_2 \otimes_R (B \otimes_S P_2) \ar[l]  & \cdots \ar[l] \\
Q_0 \otimes_R (B \otimes_S P_1)  & Q_1 \otimes_R (B \otimes_S P_1) \ar[l] & Q_2 \otimes_R (B \otimes_S P_1) \ar[l]& \cdots \ar[l] \\
Q_0 \otimes_R (B \otimes_S P_0)  & Q_1 \otimes_R (B \otimes_S P_0) \ar[l] & Q_2 \otimes_R (B \otimes_S P_0) \ar[l]  & \cdots \ar[l]
}
\]
So
\[
\hat{E}^1=\xymatrix{
\cdots \ar[d]& \cdots  \ar[d] & \cdots  \ar[d] &   \\
A \otimes_R (B \otimes_S P_2) \ar[d] & \Tor^R_1(A, B \otimes_S P_2) \ar[d] & \Tor^R_2(A, B \otimes_S P_2) \ar[d]  & \cdots  \\
A \otimes_R (B \otimes_S P_1)  \ar[d]& \Tor^R_1(A, B \otimes_S P_1) \ar[d] & \Tor^R_2(A, B \otimes_S P_1)  \ar[d]& \cdots  \\
A \otimes_R (B \otimes_S P_0)  & \Tor^R_1(A, B \otimes_S P_0)  & \Tor^R_2(A, B \otimes_S P_0)   & \cdots 
}
\]
By assumption, all but the left most column in $\hat{E}_1$ vanish. Also note that $A \otimes_R (B \otimes_S P_j) \cong (A \otimes_R B) \otimes_S P_j$. So
\[
\hat{E}^2_{p,q} \cong \begin{cases}
\Tor^S_q (A \otimes_R B, C) & \text{if } p=0,\\
0 & \text{if } p>0.
\end{cases}
\]
Thus, $\{\hat{E}^r\}$ collapses at its $E^2$-page. This implies that, as graded $\Q$-space, the $n^{th}$ homology of the total complex of double complex \eqref{eq-double-complex} is isomorphic to $\Tor^S_n (A \otimes_R B, C)$.

Now consider the spectral sequence $\{E^r\}$.  Note that
\[
E^0=\xymatrix{
\cdots \ar[d]& \cdots   \ar[d]& \cdots   \ar[d]&   \\
Q_0 \otimes_R (B \otimes_S P_2)  \ar[d]& Q_1 \otimes_R (B \otimes_S P_2) \ar[d]& Q_2 \otimes_R (B \otimes_S P_2) \ar[d]& \cdots  \\
Q_0 \otimes_R (B \otimes_S P_1)  \ar[d]& Q_1 \otimes_R (B \otimes_S P_1) \ar[d]& Q_2 \otimes_R (B \otimes_S P_1) \ar[d]& \cdots  \\
Q_0 \otimes_R (B \otimes_S P_0)  & Q_1 \otimes_R (B \otimes_S P_0) & Q_2 \otimes_R (B \otimes_S P_0) & \cdots 
}
\]
Recall that projective module are flat. By Corollary \ref{lemma-universal-projective}, 
\[
E^1=\xymatrix{
\cdots & \cdots   & \cdots   &   \\
Q_0 \otimes_R \Tor^S_2(B,C)  & Q_1 \otimes_R \Tor^S_2(B,C) \ar[l]& Q_2 \otimes_R \Tor^S_2(B,C) \ar[l]& \cdots \ar[l] \\
Q_0 \otimes_R \Tor^S_1(B,C)  & Q_1 \otimes_R \Tor^S_1(B,C) \ar[l]& Q_2 \otimes_R \Tor^S_1(B,C) \ar[l]& \cdots \ar[l] \\
Q_0 \otimes_R (B \otimes_S C)  & Q_1 \otimes_R (B \otimes_S C) \ar[l]& Q_2 \otimes_R (B \otimes_S C) \ar[l]& \cdots \ar[l]
}
\]
Therefore, $E^2_{p,q} \cong \Tor^R_p(A, \Tor^S_q(B,C))$. Moreover, recall that $\{E^r\}$ converges to the homology of the total complex of \eqref{eq-double-complex}, which is $\Tor^S_\ast (A \otimes_R B, C)$
\end{proof}

It is not too hard to prove Theorem \ref{thm-Betti} using Theorem \ref{thm-SS-base-ring}.

\begin{proof}[Proof of Theorem \ref{thm-Betti}]
$R_{B,r}:=\Q[X_2-X_1,\dots,X_m-X_1]$ is a subring of $R_B:=\Q[X_1,\dots,X_m]$. We have the isomorphism $R_B/\m \cong R_{B,r}/\m_r \otimes_{R_{B,r}} R_B/(X_1).$ Note that $R_B/(X_1)$ is a free module over $R_{B,r}$. So ${(R_B/(X_1))\otimes_{R_B}P}$ is projective over $R_{B,r}$ whenever $P$ is projective over $R_B$. Thus, 
\[
{\Tor^{R_{B,r}}_i(R_{B,r}/\m_r,(R_B/(X_1))\otimes_{R_B}P)} =0
\] 
for all $i\geq 1$. This means that $R=R_{B,r}$, $S=R_B$, $A=R_{B,r}/\m_r$ and $B= R_B/(X_1)$ satisfy the conditions required in Theorem \ref{thm-SS-base-ring}. Now apply Theorem \ref{thm-SS-base-ring} to the $R_B$-module $C=H^{\star,j,k}(B)$. From Remark \ref{remark-quotient}, we know that $H(B) \cong H_r(B)\otimes_\Q \Q[X_1]\{1,0,0\}$. In particular, $H^{\star,j,k}(B)$ is a free $\Q[X_1]$-module, and $H^{\star,j,k}(B)/X_1H^{\star,j,k}(B) \cong H_r^{\star,j,k}(B)\{1\}$, where ``$\{1\}$" means shifting the $R_{B,r}$-module grading up by $1$. Note that $R_B/(X_1)$ has the simple minimal free resolution $0 \rightarrow R_B\{2\} \xrightarrow{X_1} R_B$. So $\Tor^{R_B}_\ast(R_B/(X_1), H^{\star,j,k}(B))$ is isomorphic to the homology of the chain complex $0 \rightarrow H^{\star,j,k}(B)\{2\} \xrightarrow{X_1} H^{\star,j,k}(B)$. But $H^{\star,j,k}(B)$ is a free $\Q[X_1]$-module. So $\Tor^{R_B}_q(R_B/(X_1), H^{\star,j,k}(B))\cong 0$ for all $q\geq1$. This shows that, in our case, the $E^2$-page of the spectral sequence in Theorem \ref{thm-SS-base-ring} is supported on the degree $q=0$. Thus, this spectral sequence collapses at its $E^2$-page. Consequently, as graded $\Q$-spaces,
\begin{eqnarray*}
\Tor^{R_B}_{p}(R_B/\m, H^{\star,j,k}(B)) & \cong & \Tor^{R_B}_{p}(R_{B,r}/\m_r \otimes_{R_{B,r}} R_B/(X_1), H^{\star,j,k}(B)) \\
& \cong & \Tor^{R_{B,r}}_p (R_{B,r}/\m_r, R_B/(X_1) \otimes_{R_B}H^{\star,j,k}(B)) \\
& \cong & \Tor^{R_{B,r}}_p (R_{B,r}/\m_r, H^{\star,j,k}(B)/X_1H^{\star,j,k}(B)) \\
& \cong & \Tor^{R_{B,r}}_p (R_{B,r}/\m_r, H_r^{\star,j,k}(B))\{1\}
\end{eqnarray*}
By the definition of the Betti numbers, this proves that $\beta_B(p,q,j,k) = \beta_{B,r}(p,q-1,j,k)$.

In the case when $B$ is a knot, we have $R_{B,r}=\Q$ and $\m_r=\{0\}$. So $\beta_{B,r}(p,q,j,k)=0$ if $p\geq1$, and $\beta_{B,r}(0,q,j,k) = \dim_\Q H_r^{q,j,k}(B)=\dim_\Q \overline{H}^{q,j,k}(B)$, since $H_r(B)=\overline{H}(B)$ when $B$ is a knot.
\end{proof}

\section{Split Union}\label{sec-split}

To understand the behavior of the Betti numbers under split union, we need the following lemma.

\begin{lemma}\label{lemma-mfs-tensor}
Let $X_1,\dots,X_m,Y_1,\dots,Y_n$ be pairwise distinct variables of degree $2$, $R_X=\Q[X_1,\dots,X_m]$ and $R_Y=\Q[Y_1,\dots,Y_n]$. Assume that $M_X$ is a finitely generated graded $R_X$-module with the minimal free resolution $0 \rightarrow F_l \rightarrow F_{l-1} \rightarrow \cdots \rightarrow F_1 \rightarrow F_0$ over $R_X$, and $M_Y$ is a finitely generated graded $R_Y$-module with the minimal free resolution $0 \rightarrow G_k \rightarrow G_{k-1} \rightarrow \cdots \rightarrow G_1 \rightarrow G_0$ over $R_Y$. Then the finitely generated $R_X\otimes_\Q R_Y$-module $M_X \otimes_\Q M_Y$ has the minimal free resolution $F_\ast \otimes_\Q G_\ast$ over $R_X\otimes_\Q R_Y$, which is of the form
\begin{equation}\label{eq-mfs-tensor}
0 \rightarrow F_l \otimes_\Q G_k \rightarrow \cdots \rightarrow \bigoplus_{p_1+p_2=p} F_{p_1}\otimes_\Q G_{p_2} \rightarrow \cdots \rightarrow F_0 \otimes_\Q G_0.
\end{equation}

Denote by $\beta_{M_X}(p,q)$ the Betti number of $M_X$ over $R_X$, by $\beta_{M_Y}(p,q)$ the Betti number of $M_Y$ over $R_Y$ and by $\beta_{M_X\otimes_\Q M_Y}(p,q)$ the Betti number of $M_X\otimes_\Q M_Y$ over $R_X\otimes_\Q R_Y$. Let
\begin{eqnarray*}
\mathcal{P}_{M_X}(x,y) & = & \sum_{(p,q)\in \zed_{\geq0}\times\zed} \beta_{M_X}(p,q) \cdot x^p \cdot \left(\sum_{i\in \zed} y^{2i+q} \bn{i+m-1}{i}\right), \\
\mathcal{P}_{M_Y}(x,y) & = & \sum_{(p,q)\in \zed_{\geq0}\times\zed} \beta_{M_Y}(p,q) \cdot x^p \cdot \left(\sum_{i\in \zed} y^{2i+q} \bn{i+n-1}{i}\right), \\
\mathcal{P}_{M_X \otimes_\Q M_Y}(x,y) & = & \sum_{(p,q)\in \zed_{\geq0}\times\zed} \beta_{M_{M_X\otimes_\Q M_Y}}(p,q) \cdot x^p \cdot \left(\sum_{i\in \zed} y^{2i+q} \bn{i+m+n-1}{i}\right).
\end{eqnarray*}
Then $\mathcal{P}_{M_X \otimes_\Q M_Y}(x,y) = \mathcal{P}_{M_X}(x,y) \cdot \mathcal{P}_{M_Y}(x,y).$
\end{lemma}

\begin{proof}
It is clear that $F_\ast \otimes_\Q G_\ast$ is a minimal chain complex of graded free $R_X\otimes_\Q R_Y$-modules. So we only need to verify that it is a resolution of $M_X \otimes_\Q M_Y$. Since $\Q$ is a field, the K\"unneth Formula gives that
\[
H_p(F_\ast \otimes_\Q G_\ast) \cong \bigoplus_{p_1+p_2=p} H_{p_1}(F_\ast) \otimes_\Q H_{p_2}(G_\ast) \cong \begin{cases}
M_X \otimes_\Q M_Y & \text{if } p=0,\\
0 & \text{otherwise.}
\end{cases}
\]
It shows that $F_\ast \otimes_\Q G_\ast$ is a minimal free resolution of $M_X \otimes_\Q M_Y$ over $R_X\otimes_\Q R_Y$.

By Lemma \ref{lemma-Poincare-polynomial-general},
\begin{eqnarray*}
\mathcal{P}_{M_X}(x,y) & = & \sum_{(p,i)\in \zed_{\geq0}\times\zed} x^p \cdot y^i \cdot \dim_\Q F_p^i, \\
\mathcal{P}_{M_Y}(x,y) & = & \sum_{(p,i)\in \zed_{\geq0}\times\zed} x^p \cdot y^i \cdot \dim_\Q G_p^i,
\end{eqnarray*}
where $F_p^i$ (resp. $G_p^i$) is the homogeneous component of $F_p$ (resp. $G_p$) of degree $i$. Clearly,
\[
\left(\sum_{(p,i)\in \zed_{\geq0}\times\zed} x^p \cdot y^i \cdot \dim_\Q F_p^i\right) \cdot \left(\sum_{(p,i)\in \zed_{\geq0}\times\zed} x^p \cdot y^i \cdot \dim_\Q G_p^i\right) = \sum_{(p,i)\in \zed_{\geq0}\times\zed}  x^p \cdot y^i \cdot \dim_\Q \left( \bigoplus_{p_1+p_2=p} F_{p_1}\otimes_\Q G_{p_2}\right)^i,
\]
where $\left(\bigoplus_{p_1+p_2=p} F_{p_1}\otimes_\Q G_{p_2}\right)^i$ is the homogeneous component of $\bigoplus_{p_1+p_2=p} F_{p_1}\otimes_\Q G_{p_2}$ of degree $i$. But $M_X\otimes_\Q M_Y$ has the minimal free resolution $F_\ast \otimes_\Q G_\ast$ over $R_X \otimes_\Q R_Y$, which is of form \eqref{eq-mfs-tensor}. So, by Lemma \ref{lemma-Poincare-polynomial-general},
\[
\mathcal{P}_{M_X \otimes_\Q M_Y}(x,y) = \sum_{(p,i)\in \zed_{\geq0}\times\zed}  x^p \cdot y^i \cdot \dim_\Q \left( \bigoplus_{p_1+p_2=p} F_{p_1}\otimes_\Q G_{p_2}\right)^i.
\]
Thus, $\mathcal{P}_{M_X \otimes_\Q M_Y}(x,y) = \mathcal{P}_{M_X}(x,y) \cdot \mathcal{P}_{M_Y}(x,y).$
\end{proof}
 
The next lemma is a simple observation.

\begin{lemma}\label{lemma-pd-m-1}
Let $B$ be a closed braid with $m$ components. Then $\pd_{R_B} H(B) \leq m-1$.
\end{lemma}

\begin{proof}
By Lemma \ref{lemma-pd} and Theorem \ref{thm-Betti}, 
\begin{eqnarray*}
\pd_{R_B} H(B) & = & \max\{p~|~ \beta_B(p,q,j,k)\neq 0 \text{ for some } (q,j,k)\in \zed^3\} \\
& = & \max\{p~|~ \beta_{B,r}(p,q,j,k)\neq 0 \text{ for some } (q,j,k)\in \zed^3\} = \pd_{R_{B,r}} H_r(B).
\end{eqnarray*}
But $R_{B,r}$ is a polynomial ring of $m-1$ variables. So its global dimension is $m-1$. Thus, $\pd_{R_B} H(B) = \pd_{R_{B,r}} H_r(B) \leq m-1$.
\end{proof}

Now we are ready to prove Theorem \ref{thm-split}.

\begin{proof}[Proof of Theorem \ref{thm-split}]
We prove Part (1) first. Using the polynomial notation in Lemma \ref{lemma-mfs-tensor}, we have
\begin{eqnarray*}
\mathcal{P}_{B_1} (x,y,a,b) & = & \sum_{(j,k)\in \zed^2} a^j \cdot b^{\frac{k-j}{2}} \cdot \mathcal{P}_{H^{\star,j,k}(B_1)}(x,y), \\
\mathcal{P}_{B_2} (x,y,a,b) & = & \sum_{(j,k)\in \zed^2} a^j \cdot b^{\frac{k-j}{2}} \cdot \mathcal{P}_{H^{\star,j,k}(B_2)}(x,y), \\
\mathcal{P}_{B_1\sqcup B_2} (x,y,a,b) & = & \sum_{(j,k)\in \zed^2} a^j \cdot b^{\frac{k-j}{2}} \cdot \mathcal{P}_{H^{\star,j,k}(B_1\sqcup B_2)}(x,y).
\end{eqnarray*}
By isomorphism \eqref{eq-H-tensor},
\[
H^{\star,j,k}(B_1\sqcup B_2) \cong \bigoplus_{j_1+j_2=j-1,~k_1+k_2 =k+1} H^{\star,j_1,k_1}(B_1) \otimes_\Q H^{\star,j_2,k_2}(B_2).
\]
Thus, by Lemma \ref{lemma-mfs-tensor},
\[
\mathcal{P}_{H^{\star,j,k}(B_1\sqcup B_2)}(x,y) = \sum_{j_1+j_2=j-1,~k_1+k_2 =k+1} \mathcal{P}_{H^{\star,j_1,k_1}(B_1)}(x,y) \cdot \mathcal{P}_{H^{\star,j_2,k_2}(B_2)}(x,y).
\]
Therefore,
\begin{eqnarray*} 
&& \mathcal{P}_{B_1\sqcup B_2} (x,y,a,b) \\
& = & \sum_{(j,k)\in \zed^2} a^j \cdot b^{\frac{k-j}{2}} \cdot \mathcal{P}_{H^{\star,j,k}(B_1\sqcup B_2)}(x,y) \\
& = & ab^{-1}\sum_{(j,k)\in \zed^2} a^{j-1} \cdot b^{\frac{(k+1)-(j-1)}{2}} \cdot \left(\sum_{j_1+j_2=j-1,~k_1+k_2 =k+1} \mathcal{P}_{H^{\star,j_1,k_1}(B_1)}(x,y) \cdot \mathcal{P}_{H^{\star,j_2,k_2}(B_2)}(x,y)\right) \\
& = &  ab^{-1}\left(\sum_{(j_1,k_1)\in \zed^2} a^{j_1} \cdot b^{\frac{k_1-j_1}{2}} \cdot \mathcal{P}_{H^{\star,j_1,k_1}(B_1)}(x,y)\right) \cdot \left(\sum_{(j_2,k_2)\in \zed^2} a^{j_2} \cdot b^{\frac{k_2-j_2}{2}} \cdot \mathcal{P}_{H^{\star,j_2,k_2}(B_1)}(x,y)\right) \\
& = & ab^{-1}\mathcal{P}_{B_1} (x,y,a,b) \cdot \mathcal{P}_{B_2} (x,y,a,b).
\end{eqnarray*}
Combining this and Lemma \ref{lemma-pd}, we get
\begin{eqnarray*}
\pd_{R_{B_1 \sqcup B_2}} H(B_1 \sqcup B_2) & = & \deg_x \mathcal{P}_{B_1\sqcup B_2} (x,y,a,b) \\
& = & \deg_x \mathcal{P}_{B_1} (x,y,a,b) + \deg_x \mathcal{P}_{B_2} (x,y,a,b) = \pd_{R_{B_1}} H(B_1)+\pd_{R_{B_2}} H(B_2).
\end{eqnarray*}
This completes the proof of Part (1).

For Part (2), without loss of generality, assume that $B$ is an $m$-component closed braid that is $n$-split. Then there are $n$ closed braid $B_1,\dots,B_n$ such that $B = B_1 \sqcup \cdots \sqcup B_n$. Denote by $m_l$ the number of components of $B_l$. Note that $\sum_{l=1}^n m_l =m$. By Part (1) and Lemma \ref{lemma-pd-m-1}, we have 
\[
\pd_{R_B} H(B) = \sum_{l=1}^n \pd_{R_{B_l}} H(B_l) \leq \sum_{l=1}^n (m_l-1) =m-n.
\]
This completes the proof of Theorem \ref{thm-split}.
\end{proof}

\section{The Hopf Link}\label{sec-hopf}

In this section, we compute the Betti numbers of the middle HOMFLYPT homology of the positive Hopf link $B_1$ and verify Example \ref{eg-hopf}.

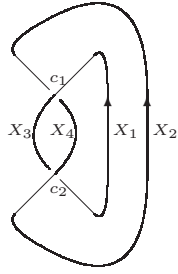
\begin{figure}[ht]
\[
\setlength{\unitlength}{1pt}
\begin{picture}(75,125)(-20,-32.5)

\put(-15,0){\line(1,1){15}}

\qbezier(0,15)(15,30)(2,43)

\put(-2,47){\line(-1,1){13}}

\put(15,0){\line(-1,1){13}}

\put(0,45){\line(1,1){15}}

\qbezier(-2,17)(-15,30)(0,45)

\qbezier(15,60)(20,65)(20,45)

\qbezier(15,0)(20,-5)(20,15)

\put(20,15){\vector(0,1){30}}

\qbezier(-15,60)(-20,65)(0,75)

\qbezier(0,75)(35,92.5)(35,45)

\qbezier(-15,0)(-20,-5)(0,-15)

\qbezier(0,-15)(35,-32.5)(35,15)

\put(35,15){\vector(0,1){30}}

\put(-18,30){\tiny{$X_3$}}

\put(-2,30){\tiny{$X_4$}}

\put(22,30){\tiny{$X_1$}}

\put(37,30){\tiny{$X_2$}}

\put(-2,50){\tiny{$c_1$}}

\put(-2,8){\tiny{$c_2$}}

\end{picture} 
\]
\caption{The Hopf link $B_1$}\label{fig-hopf}

\end{figure}

Figure \ref{fig-hopf} is a standard diagram of $B_1$. The variables assigned to the edges of this diagram are $X_1,X_2,X_3,X_4$ as shown in Figure \ref{fig-hopf}. The base ring $R=R(B_1)$ of the double chain complex $C_0(B_1)$ is $R=R(B_1)=\Q[X_1,X_2,X_3,X_4]/(X_1+X_2-X_3 -X_4) \cong \Q[X_1,X_2,X_3]$. The double chain complexes associated to the two crossings $c_1$ and $c_2$ of $B_1$ are
\[
C_0(c_1) = \xymatrix{
R\{0,-2,0\} \ar[rrr]^{X_2-X_3} &&& R\{0,0,0\} \\
R\{2,-2,-2\} \ar[rrr]^{(X_2-X_3)(X_1-X_3)} \ar[u]^{X_1-X_3} &&& R\{0,0,-2\} \ar[u]^{1}
}
\]
and 
\[
C_0(c_2) = \xymatrix{
R\{0,-2,0\} \ar[rrr]^{X_3-X_2} &&& R\{0,0,0\} \\
R\{2,-2,-2\} \ar[rrr]^{-(X_2-X_3)(X_1-X_3)} \ar[u]^{X_1-X_3} &&& R\{0,0,-2\}.\ar[u]^{1}
}
\]
So the double chain complex $C_0(B_1)$ is 
\[
C_0(B_1) = \xymatrix{
R\{0,-4,0\} \ar[rr]^{d_+^{-4,0}} &&
{\left.%
\begin{array}{l}
  R\{0,-2,0\}\\
  \oplus \\
  R\{0,-2,0\}
\end{array}%
\right.} \ar[rr]^{d_+^{-2,0}} && R\{0,0,0\} \\
{\left.%
\begin{array}{l}
  R\{2,-4,-2\}\\
  \oplus \\
  R\{2,-4,-2\}
\end{array}%
\right.} \ar[rr]^{d_+^{-4,-2}} \ar[u]^{d_v^{-4,-2}} &&
{\left.%
\begin{array}{l}
  R\{0,-2,-2\} \\
  \oplus \\
  R\{2,-2,-2\}\\
  \oplus \\
  R\{2,-2,-2\} \\
  \oplus \\
  R\{0,-2,-2\}
\end{array}%
\right.} \ar[rr]^{d_+^{-2,-2}} \ar[u]^>>>>{d_v^{-2,-2}} &&
{\left.%
\begin{array}{l}
  R\{0,0,-2\}\\
  \oplus \\
  R\{0,0,-2\}
\end{array}%
\right.} \ar[u]^{d_v^{0,-2}} \\
R\{4,-4,-4\} \ar[rr]^{d_+^{-4,-4}} \ar[u]^{d_v^{-4,-4}}
&& 
{\left.%
\begin{array}{l}
  R\{2,-2,-4\}\\
  \oplus \\
  R\{2,-2,-4\}
\end{array}%
\right.} \ar[rr]^{d_+^{-2,-4}} \ar[u]^>>>>{d_v^{-2,-4}} &&
R\{0,0,-4\}, \ar[u]^{d_v^{0,-4}}
}
\]
where the horizontal chain maps are 
\begin{eqnarray*}
d_+^{-2,0} & = & (X_2-X_3,X_3-X_2), \\
d_+^{-4,0} & = & {\left(%
\begin{array}{c}
  X_2-X_3\\
  X_2-X_3
\end{array}%
\right)}, \\
d_+^{-2,-2} & = & {\left(%
\begin{array}{cccc}
 X_2-X_3 & -(X_2-X_3)(X_1-X_3) & 0 & 0 \\
  0& 0& (X_2-X_3)(X_1-X_3) & X_3-X_2
\end{array}%
\right)}, \\
d_+^{-4,-2} & = & {\left(%
\begin{array}{cc}
  (X_2-X_3)(X_1-X_3) & 0 \\
  X_2-X_3 & 0\\
  0 & X_2-X_3\\
  0 & (X_2-X_3)(X_1-X_3)
\end{array}%
\right)}, \\
d_+^{-2,-4} & = & ((X_2-X_3)(X_1-X_3), -(X_2-X_3)(X_1-X_3)), \\
d_+^{-4,-4} & = & {\left(%
\begin{array}{c}
  (X_2-X_3)(X_1-X_3) \\
  (X_2-X_3)(X_1-X_3)
\end{array}%
\right)},
\end{eqnarray*}
and the vertical chain maps are
\begin{eqnarray*}
d_v^{0,-2} & = & (1,1) \\
d_v^{-2,-2}& = & {\left(%
\begin{array}{cccc}
 1 & 0 & X_1-X_3 & 0 \\
 0 & X_1-X_3 & 0 & 1
\end{array}%
\right)}, \\
d_v^{-4,-2} & = & (X_1-X_3,X_1-X_3), \\
d_v^{0,-4} & = & {\left(%
\begin{array}{c}
  1\\
  -1
\end{array}%
\right)}, \\
d_v^{-2,-4} & = & {\left(%
\begin{array}{cc}
  X_1-X_3 & 0 \\
  0 & 1 \\
  -1 & 0 \\
  0 & X_3-X_1
\end{array}%
\right)}, \\
d_v^{-4,-4} & = & {\left(%
\begin{array}{c}
  X_1-X_3\\
  X_3-X_1
\end{array}%
\right)}.
\end{eqnarray*}

Thus, the homology of $C_0(B_1)$ with respect to $d_+$ is
\begin{eqnarray*}
&& H(C_0(B_1),d_+) \cong\\
&& \xymatrix{
0  &
{\left(%
\begin{array}{c}
  1\\
  1 
\end{array}%
\right)} \cdot R/(X_2-X_3)  & R/(X_2-X_3) \\
0 \ar[u]^{d_v^{-4,-2}} &
{\left(%
\begin{array}{c}
  0 \\
  0 \\
  1\\
  X_1-X_3
\end{array}%
\right)\cdot R/(X_2-X_3)} \oplus  {\left(%
\begin{array}{c}
  X_1-X_3 \\
  1 \\
  0\\
  0
\end{array}%
\right)\cdot R/(X_2-X_3)} \ar[u]^>>>>{d_v^{-2,-2}} &
{\left.%
\begin{array}{l}
  R/(X_2-X_3)\\
  \oplus \\
  R/(X_2-X_3)
\end{array}%
\right.} \ar[u]^{d_v^{0,-2}} \\
0  \ar[u]^{d_v^{-4,-4}}
& 
{\left(%
\begin{array}{c}
  1\\
  1
\end{array}%
\right)} \cdot R/((X_2-X_3)(X_1-X_3)) \{2\} \ar[u]^>>>>{d_v^{-2,-4}} &
R/((X_2-X_3)(X_1-X_3)), \ar[u]^{d_v^{0,-4}}
}
\end{eqnarray*}
where we omit the second and third $\zed$-gradings of the homology since these are represented by the position of the term in the diagram. Also, note that the first $\zed$-grading of the middle term of the bottom row is shifted up by $2$.

Now we take homology with respect to $d_v$. This gives 
\[
H(H(C_0(B_1),d_+),d_v) \cong \xymatrix{
0 & \textcolor{red}{R/(X_2-X_3,X_1-X_3)} & 0 \\
0 & 0 & 0 \\
0 & R/(X_1-X_3)\{2\} & R/(X_1-X_3).
}
\]
Note that $R/(X_1-X_3) \cong \Q[X_1,X_2]=R_{B_1}$. So
\[
H^{\star,j,k}(H(C_0(B_1),d_+),d_v) \cong \begin{cases}
\textcolor{red}{R_{B_1}/(X_1-X_2)} & \text{if } (j,k)=(-2,0), \\
R_{B_1} & \text{if } (j,k)=(0,-4), \\
R_{B_1}\{2\} & \text{if } (j,k)=(-2,-4), \\
0 & \text{otherwise.}
\end{cases}
\]

Note that the writhe of $B_1$ is $w=2$ and $B_1$ has $2$ strands. So, by equation \eqref{eq-H-def},
\[
H^{\star,j,k}(B_1) = H^{\star,j-3,k-1}(H(C_0(B_1),d_+),d_v)  \cong \begin{cases}
\textcolor{red}{R_{B_1}/(X_1-X_2)} & \text{if } (j,k)=(1,1), \\
R_{B_1} & \text{if } (j,k)=(3,-3), \\
R_{B_1}\{2\} & \text{if } (j,k)=(-1,-3), \\
0 & \text{otherwise.}
\end{cases}
\]
Using Lemma \ref{lemma-mfs-Betti}, one gets
\[
\beta_{B_1}(p,q,j,k) = \begin{cases}
1 & \text{if } (p,q,j,k)=\textcolor{red}{(1,2,1,1)}, (0,0,1,1),(0,0,3,-3),(0,2,-1,-3), \\
0 & \text{otherwise.}
\end{cases}
\]
In particular, the projective dimension of $H(B_1)$ over $R_{B_1}$ is $1$ by Lemma \ref{lemma-pd}. Therefore, by Theorem \ref{thm-split}, the positive Hopf link $B_1$ does not split.

\section{Projective Dimension of the $\mathfrak{sl}(N)$ Homology}\label{sec-sl-N}

The proof of Lemma \ref{lemma-sl-N-proj-dim} is quite straightforward. But, to state it, we need to recall the relation between the projective dimension and regular sequences via the depth. First, we recall the well-known Auslander-Buchsbaum Formula, which can be found in for example \cite[Section 15]{Peeva-graded-syzygies}.

\begin{theorem}[Auslander-Buchsbaum Formula]\label{thm-Auslander-Buchsbaum}
Let $R=\Q[X_1,\dots,X_m]$, graded so that which $X_j$ is homogeneous of degree $2$. Assume that $M$ is a finitely generated graded $R$-module. Then $\pd_R M = m-\mathrm{depth}(M)$, where $\mathrm{depth}(M)$ is the depth of $M$ over $R$ with respect to the maximal homogeneous ideal $\mathfrak{m}=(X_1,\dots,X_m)$ of $R$.
\end{theorem}

Next we recall the definition of regular sequences in \cite[Section 14]{Peeva-graded-syzygies}.

\begin{definition}\label{def-regular}
Let $R=\Q[X_1,\dots,X_m]$, graded so that which $X_j$ is homogeneous of degree $2$. Assume that $M$ is a finitely generated graded $R$-module. 

An element $f \in R$ is a non-zero divisor on $M$ if $fu\neq 0$ for every non-zero element $u$ of $M$. Otherwise, $f$ is called a zero divisor on $M$.

A sequence $f_1,\dots,f_q\in R$ is an $M$-regular sequence if 
\begin{itemize}
	\item $(f_1,\dots,f_q)M \neq M$,
	\item for every $1\leq i \leq q$, $f_i$ is a non-zero divisor on the module $M/(f_1,\dots,f_{i-1})M$.
\end{itemize}
\end{definition}

The following relation between the depth and regular sequences is stated in \cite[Proposition 20.1]{Peeva-graded-syzygies} and proved in \cite[Propositions 1.5.11 and 1.5.12]{Bruns-Herzog}.

\begin{proposition}\label{prop-depth-regular-sequences}
Let $R=\Q[X_1,\dots,X_m]$, graded so that which $X_j$ is homogeneous of degree $2$. Assume that $M$ is a finitely generated graded $R$-module. If the depth of $M$ over $R$ with respect to its maximal homogeneous ideal is $s$, then there exists an $M$-regular sequence $f_1,f_2,\dots,f_s$ of homogeneous elements of $R$.
\end{proposition}

Now we are ready to prove Lemma \ref{lemma-sl-N-proj-dim}.

\begin{proof}[Proof of Lemma \ref{lemma-sl-N-proj-dim}]
Recall that $H_N(B)$ is a graded $R_B$-module that is finite dimensional over $\Q$. This implies that $M$ is finitely generated over $R_B$. Moreover, this also implies that any homogeneous element of $R_B$ of positive degree is a zero divisor on $H_N(B)$. Thus, there are no $H_N(B)$-regular sequences of homogeneous elements of $R_B$ of any positive length. By Proposition \ref{prop-depth-regular-sequences}, this implies that $\mathrm{depth}(H_N(B))=0$, where the depth is  over $R_B$ and with respect to its maximal homogeneous ideal. Now the Auslander-Buchsbaum Formula gives that $\pd_{R_B} H_N(B) = m-\mathrm{depth}(H_N(B))=m$.
\end{proof}

\end{document}